\documentclass[12pt]{amsart}

\numberwithin{equation}{section}

\usepackage{amsmath,amsthm,amsfonts,eucal,}
\usepackage{xypic}
\usepackage{graphicx}
\usepackage{amssymb}
\usepackage{color}

\def\cb{{\mathcal B}}

\def\ch{{\mathcal H}}

\def\ck{{\mathcal K}}

\def\ga{{\mathfrak A}}

\def\bc{{\mathbb C}}

\def\bn{{\mathbb N}}

\def\bt{{\mathbb T}}

\def\bz{{\mathbb Z}}

\def\a{\alpha}

\def\r{\rho}
\def\s{\sigma}

\def\om{\omega} \def\Om{\Omega}

\newtheorem{thm}{Theorem}[section]

\newtheorem{lem}[thm]{Lemma}
\newtheorem{cor}[thm]{Corollary}
\newtheorem{prop}[thm]{Proposition}

\theoremstyle{definition}
\newtheorem{rem}[thm]{Remark}

\def\dim{\mathop{\rm dim}}
\def\min{\mathop{\rm min}}

\def\di{\mathop{\rm d}\!}
\def\di{{\rm d}}

\begin{document}

\title[Weakly-monotone $C^*$-algebras as Exel-Laca algebras]
{Weakly-monotone $C^*$-algebras as Exel-Laca algebras}
\author{Vitonofrio Crismale}
\address{Vitonofrio Crismale\\
Dipartimento di Matematica\\
Universit\`{a} degli studi di Bari\\
Via E. Orabona, 4, 70125 Bari, Italy}
\email{\texttt{vitonofrio.crismale@uniba.it}}

\author{Simone Del Vecchio}
\address{Simone Del Vecchio\\
Dipartimento di Matematica\\
Universit\`{a} degli studi di Bari\\
Via E. Orabona, 4, 70125 Bari, Italy}
\email{\texttt{simone.delvecchio@uniba.it}}

\author{Stefano Rossi}
\address{Stefano Rossi\\
Dipartimento di Matematica\\
Universit\`{a} degli studi di Bari\\
Via E. Orabona, 4, 70125 Bari, Italy}
\email{\texttt{stefano.rossi@uniba.it}}

\author{Janusz Wysocza\'nski}
\address{Janusz Wysocza\'nski\\
Department of Mathematics\\
Wroclaw University\\
plac Grunwaldzki 2, 50--384 Wroclaw, Poland}
\email{\texttt{jwys@math.uni.wroc.pl }}


\begin{abstract}

An abstract characterization of  weakly monotone
$C^*$-algebras, namely the concrete $C^*$-algebras generated
by creators and annihilators acting on the so-called weakly monotone
Fock spaces, is given  in terms of (quotients of) suitable Exel-Laca algebras.
The weakly monotone $C^*$-algebra indexed by $\bn$ is shown to be a type-I
$C^*$-algebra and its representation theory is entirely determined, whereas
the  weakly monotone $C^*$-algebra indexed by $\bz$ is shown not to be of type $I$.

\end{abstract}

\maketitle

\section{Introduction}

One of the most relevant notions in Classical Probability,  stochastic independence for $\s$-algebras and random variables
is intimately tied to the possibility of factorizing the underlying probability space into a product space.
In addition,  factorizing the probability space amounts to factorizing the Hilbert space of all its
square-integrable random variables into a Hilbert tensor product of suitable subspaces.
Quite remarkably, this is the only notion of independence encountered in Classical Probability.
Noncommutative Probability, by contrast, features several different ways in which stochastic variables can be independent,  and each of these notions leads to its own central limit theorem.
A thorough account
of the various forms of non-commutative independence can be found in {\it e.g.} \cite{Mu}.\\
For the purposes of the present paper, however, we can limit ourselves to mentioning the so-called
monotone independence, which was introduced in \cite{Lu1} and in \cite{Mu2}.
It was soon realized that the central limit distribution associated with monotone independence is
the arcsine law. This result was obtained in \cite{Mu2} in a $C^*$-algebraic setting for a sequence of
self-adjoint random variables which are monotone independent with respect to a given state.
An interesting application was also given in the aforementioned paper by looking at the so-called position operators acting on a suitably contrived Hilbert space, the monotone Fock space.\\
Shortly after, the so-called weakly monotone Fock space was defined in \cite{W05} along with a countable family of
annihilators and creators (further generalizations of weakly monotone independence such as the so-called bm-independence were introduced and studied in \cite{W07} and  \cite{W10}).
In the more recent paper \cite{CGW}, it is shown that the  $C^*$-subalgebras generated by a single annihilator are a family of monotone independent algebras with respect to the vacuum state, that
 any self-adjoint position operator has the Wigner distribution in the vacuum,
and that the distribution of any finite sum
of such position operators, which is given by  the monotone convolution \cite{Mupreprint} of the Wigner distribution with itself, is absolutely continuous with a support that can be described explicitly.
The distribution of the so-called non-symmetric position operator is no longer given by the Wigner law, being in fact a free Meixner \cite{SY}, as proved in \cite{CGW2}, where its monotone convolution is also studied.

As well as providing  models where monotone independence holds, the concrete $C^*$-algebras generated by
annihilators and creators on the monotone Fock spaces mentioned above are also worth
studying as $C^*$-algebras in their own right. For instance, in the recent paper
 \cite{CDR} the so-called (strictly) monotone $C^*$-algebra is shown to
be a universal $C^*$-algebra which can described abstractly in terms of generators and relations.
Among other things, this turns out to be useful to simplify the task of exhibiting its approximately
finite-dimensional $C^*$-algebra structure.\\
When weakly monotone Fock spaces are considered, the corresponding
concrete $C^*$-algebras are still liable to being characterized abstractly. One
marked difference, though, is that the approximately finite-dimensional structure can by no means be obtained, in that
the localized subalgebras are no longer matrix algebras, being more or less akin to the
Coburn algebra instead. This is the case of  weakly monotone $C^*$-algebras with  finitely many  generators, which in \cite{GW} are shown  to be isomorphic with (quotients of) suitable Cuntz-Krieger algebras.
These are a much studied class of $C^*$-algebras introduced in the celebrated paper \cite{C1981} as
universal $C^*$-algebras generated by nonzero partial isometries $S_i$, indexed by a \textit{finite} set $\Sigma$, with relations for their support projections $Q_i=S_i^*S_i$ and range projections $P_i=S_iS_i^*$, known as Cuntz-Krieger relations, given by a finite matrix $\displaystyle A=(a_{ij})_{i,j\in \Sigma}$, with entries $a_{ij}\in \{0, 1\}$:
\begin{eqnarray*}
&& P_iP_j = 0\quad \text{if} \quad i\neq j \\
&& Q_i=\sum_{j\in\Sigma} a_{ij}P_j\,.
\end{eqnarray*}
In this paper, we turn our attention to weakly monotone $C^*$-algebras with infinitely
many generators. Characterizing such algebras in
terms of generators and relations poses a slightly harder problem, in that relations involving
infinitely many terms may possibly show up.
One way to overcome this problem is  to make use of so-called Exel-Laca algebras, see \cite{EL0}, which can be
thought of generalized Cuntz-Krieger algebras arising from infinite matrices whose entries are all zeros and ones. That is exactly what we do to obtain Theorem \ref{charZ}, where
the weakly monotone $C^*$-algebra with generators indexed by $\bz$, $O_A$, is presented as the Exel-Laca algebra
associated with the infinite matrix  $A=(a_{i,j})_{i,j\in\bz}$ with $a_{i,j}=1$ if $i\geq j$ and $a_{i,j}=0$ otherwise.
Another possible strategy would be to apply the more recent approach developed in \cite{BdC}.\\
Interestingly, when the generators are indexed by $\bn$, the corresponding weakly monotone $C^*$-algebra, $O_{A_\bn}$,
can be handled more directly in terms of generators and relations because each defining relation only involves
finitely many terms. In fact, Proposition \ref{faithfulrep} and Corollary \ref{weaklyrel} provide a space-free
description of $O_{A_\bn}$ respectively as a quotient of an inductive limit of
Cuntz-Krieger $C^*$-algebras and the universal $C^*$-algebra with given generators and relations.\\
Changing the index set from $\bz$ to $\bn$ causes even more striking consequences to occur.
To begin with, the projection onto the vacuum vector does belong to (a natural representation of) $O_{A_\bn}$, whereas it does not
sit in $O_{A}$, Proposition \ref{vacumnotalg}.
Whether the projection onto the vacuum belongs or not to our weakly monotone
$C^*$-algebras has deep ramifications in the resulting representation theory.
Precisely, $O_{A}$ is not a $C^*$-algebra of type $I$, Corollary \ref{nottypeI}, meaning
$O_{A}$ has representations of all types.
By contrast, $O_{A_\bn}$ turns out to be a $C^*$-algebra of type $I$, Corollary \ref{typeI}, as we prove after
determining all irreducible representations, Theorem \ref{catalogue}.
Moreover, in Theorem \ref{dirin} we show how every representation $\pi$ of $O_{A_\bn}$ on a separable
Hilbert space may be decomposed into a direct integral of irreducible representations in a rather canonical
and explicit fashion, which is determined by the analogous decomposition of $\pi(s_0)$, where $s_0$ is the only normal
generator of $O_{A_\bn}$.
Lastly, substituting the index set $\bn$ for $\bz$ prevents the shift automorphism $\tau$ on $O_A$, whose ergodic properties are thorougly discussed in Subsection \ref{erg}, from acting on $O_{A_\bn}$.\\
Finally, the property of being type $I$ does not depend only on the chosen index set, for it also depends on
the monotone relations as well. In fact, in
Corollary \ref{nottypeIanti} we show that the so-called anti-monotone $C^*$-algebra on $\bn$ fails to be of type $I$.

\section{Weakly monotone $C^*$-algebra on $\bz$ }
We consider the Hilbert space $\ch: =\ell^2(\bz)$ endowed with its canonical
orthonormal basis $\{e_k: k\in\bz\}$. The weakly monotone Fock space on
$\ell^2(\bz)$,   which we denote  by $\mathcal{F}_{wm}(\ch)$,
is the closed subspace of the
full Fock space generated by simple tensors of the type
$$e_{i_1}\otimes e_{i_2}\otimes\cdots\otimes e_{i_k}\, \quad \textrm{with}\,\,\,
k\in\bn\,\,\, \textrm{and}\,\,\, i_1\geq i_2\geq\cdots\geq i_k\,.$$
The length of a simple tensor as above is just the number $k$, which will be sometimes referred to
as the number of particles of the vector itself.
Obviously, having length equal to zero corresponds to considering the vacuum vector $\Om$.
For every $k\geq 0$, we denote by $H_k\subset\mathcal{F}_{wm}(\ch)$ the closed subspace
generated by all possible simple tensors whose length is $k$.
The subspace $H_k$ will often be referred to as the $k$-particle subspace.
Notice that $H_0$ is the one-dimensional subspace spanned by $\Om$.
Observe that $\mathcal{F}_{wm}(\ch)$ decomposes into a direct sum as
$$\mathcal{F}_{wm}(\ch)=\bigoplus_{k=0}^\infty H_k\,.$$
We next recall how weakly monotone creators and annihilators are defined on $\mathcal{F}_{wm}(\ch)$.
For every $i\in\bz$, we consider the monotone creator $A^\dagger_i$ which acts by creating a particle in the state
$e_i$ whenever this is possible. More explicitly, the action of $A^\dagger_i$ on the orthonormal basis of
$\mathcal{F}_{wm}(\ch)$ is $A^\dagger_i \Om= e_i$ and
\begin{equation*}
A^\dagger_i\,e_{i_1}\otimes e_{i_2}\otimes\cdots\otimes e_{i_k}:=\left\{
\begin{array}{ll}
e_i\otimes e_{i_1}\otimes e_{i_2}\otimes\cdots\otimes e_{i_k}& \text{if}\,\, i\geq i_1 \\
0 & \text{otherwise}, \\
\end{array}
\right.
\end{equation*}
The corresponding annihilator $A_i$ is nothing but the adjoint of $A^\dagger_i$, that is
$A_i \Om=0$ and
\begin{equation*}
A_i\,e_{i_1}\otimes e_{i_2}\otimes\cdots\otimes e_{i_k}:=\delta_{i, i_1} e_{i_2}\otimes\cdots\otimes e_{i_k}\,
\end{equation*}
where $\delta_{i, j}$ is the Kronecker symbol.\\
By their very definition, creators and annihilators are partial isometries. Furthermore, the ranges
of the creators are mutually orthogonal, that is
\begin{equation*}
\label{comrul}
\begin{array}{ll}
A_i A^\dagger_j=0 & \text{if}\,\, i\neq j\,. \\
 \end{array}
\end{equation*}
As a consequence of the structure of the weakly monotone Fock space, we also have
\begin{equation*}
\label{comrul}
\begin{array}{ll}
  A^\dagger_iA^\dagger_j=A_jA_i=0 & \text{if}\,\, i< j\, . \\

\end{array}
\end{equation*}
In addition, the initial and final projections of our partial isometries satisfy the relation
\begin{equation*}
\label{comrul2}
A_iA^\dag_i=I-\sum_{k> i} A^\dag_k A_k\,\,\,\,\textrm{for every}\, i\in\bz\,,
\end{equation*}
where the convergence of the series is understood in the strong operator topology, and $I$ is the identity operator on the weakly monotone Fock space.\\

As well as satisfying the relations
above, the weakly monotone operators also satisfy a couple of equalities which are very much like
Cuntz-Krieger relations and which will play a role when one wants  to characterize  the concrete
$C^*$-algebra generated by these operators as a universal algebra defined in terms of generators
and relations. With this aim in mind, let us introduce a suitable infinite matrix $A$ whose entries are all
zeros or ones. This is given by $A=(a_{i, j})_ {i,j\in\bz}$ with $a_{i, j}=1$ if $i\geq j$ and $a_{i, j}=0$ otherwise.
The extra relations we alluded to are
\begin{align*}
&A_iA_i^\dag A^\dag_ j= a_{i, j}A_j^\dag\,, \quad \textrm{for all}\,\,i,j\in\bz\\
&A_iA^\dag_i+ A^\dag_{i+1} A_{i+1} = A_{i+1} A^\dag_{ i+1}\,,\quad \textrm{for all}\,\,i\in\bz
\end{align*}
In particular, we also have the inequality 
\begin{equation}\label{ineq}
 A_{i} A^\dag_{ i}\leq A_kA^\dag_k\,,\,\,\textrm {for all}\,\, i\leq k\,.
\end{equation}

\medskip

We denote by $\mathcal{W}_0$ the concrete $*$-algebra generated by the countable
set $\{A_i :i\in\bz\}$ of the annihilators acting on the weakly monotone Fock space
$\mathcal{F}_{wm}(\ch)$. As happens with the monotone $*$-algebra, see \cite{CFG}, it is also possible to exhibit a Hamel
basis for $\mathcal{W}_0$  quite explicitly. To this end, we denote by $\Gamma$ the
(countable) set of all words in the generators of the type
$$(A^\dag_{i_1})^{k_1}\cdots (A^\dag_{i_m})^{k_m}A_{j_1}^{l_1}\cdots A_{j_n}^{l_n}\,,$$
with $i_1 > \cdots > i_m$ and $j_1 < \cdots < j_n$, as $k_1, \ldots, k_m$ and $l_1, \ldots l_n$ vary
in $\bn$, where at least one between $n$ and $m$ is not zero.
The subset $\{A^\dag_iA_i : i\in\bz\}\subset\Gamma$ is clearly in bijection with $\bz$. We
denote by $\Lambda$ its complement in $\Gamma$.
\begin{prop}\label{Hamel}
The words in $\Lambda\bigcup\{A_i A^\dag_i:i\in\bz\}$ make up a Hamel basis for $\mathcal{W}_0$.
\end{prop}
The proof of the above result can be done in much the same way as the monotone case and is found
 in \cite{C}.\\
The norm closure of $\mathcal{W}_0$ in $\cb(\mathcal{F}_{wm}(\ch))$
will be
denoted by $\mathcal{W}$, which will be referred to as the concrete weakly monotone $C^*$-algebra. Arguing as in \cite[Proposition 5.8]{CFL}, one sees that the identity operator $I$ on $\mathcal{F}_{wm}(\ch)$ does not belong to $\mathcal{W}$. We then denote by $\widetilde{\mathcal{W}}$ the \emph{unital} $C^*$-subalgebra of $\cb(\mathcal{F}_{wm}(\ch))$
generated by $\{A_i :i\in\bz\}$,  in other words we have $\widetilde{\mathcal{W}}=
\{W+\lambda I: W\in\mathcal{W}, \lambda\in\bc\}$. We now move on to address both $\mathcal{W}$ and $\widetilde{\mathcal{W}}$.
We point out here that $\mathcal{W}$ is an irreducible $C^*$-subalgebra of $\cb(\mathcal{F}_{wm}(\ch)$. This can be seen easily, but it also
a straightforward consequence of Proposition \ref{selfad}, where we show that even the subalgebra generated by the set
$\{A_i+A^\dag_i: i\in\bz \}$ is irreducible.\\
We start by showing that the monotone $C^*$-algebra (see {\it e.g.} \cite{CFL, CDR}), in which the squares of all creators
are zero, cannot be obtained
as a quotient of the weakly monotone $C^*$-algebra, for the ideal generated by all powers
$A_i^2$ is  in fact the whole algebra.

\begin{prop}
The closed two-sided ideal $I\subset \mathcal{W}$ generated by the
set $\{A_i^2 : i\in\bz \}$ coincides with the whole $C^*$-algebra $\mathcal{W}$.
\end{prop}
\begin{proof}
The statement amounts to proving that the quotient $\mathcal{W}/I$ is trivial.
Denote by $\r:\mathcal{W}\rightarrow \mathcal{W}/I$ the canonical projection onto the
quotient. Multiplying the equality
$$\r(A^\dag_{i+1})\r(A_{i+1}) = -\r(A_{i})\r(A^\dag_{i})+\r(A_{i+1})\r(A^\dag_{i+1})$$
by $\r(A_{i+1})\r(A^\dag_{i+1})$ on the right and taking into account \eqref{ineq}, we get
$\r(A_i)\r(A^\dag_i) = \r(A_{i+1})\r(A^\dag_{1+1})$ for
every $i\in \bz$.\\
This implies $\r(A^\dag_{i+1})\r(A_{i+1}) = 0$, hence $\r(A_{i+1})$ = 0 for every
$i\in\bz$, thus $\mathcal{W}/I=0$.
\end{proof}
\begin{rem}
The property stated above continues to hold at the $*$- algebra level: the
 two-sided ideal $I\subset \mathcal{W}_0$ generated by the set $\{A_i^2 : i\in\bz\}$
is again equal to $\mathcal{W}_0$. Indeed, one can argue as in the proof above
to get  $\r(A^\dag_{i+1})\r(A_{i+1}) = 0$ (here $\r$ is the canonical projection
of $\mathcal{W}_0$ onto $\mathcal{W}_0/I$. But then the equality
$\r(A_i)\r(A^\dag_ j)\r(A_j) = \delta_{i, j}\r(A_j)$
leads directly to $\r(A_i) = 0$ for every $i\in\bz$.
\end{rem}

\subsection{Space-free characterization of the weakly monotone $C^*$-algebra}
We are now ready to give $\mathcal{W}$ (and $\widetilde{\mathcal{W}}$)  a space-free
characterization as an abstract $C^*$-algebra in terms of the
$C^*$-algebras $O_A$ (or $\widetilde{O_A}$) associated with a possibly infinite matrix $A$. This class
of $C^*$-algebras  was introduced and studied by
Exel and Laca in \cite{EL0} as a wide generalization of the Cuntz-Krieger algebras.\\
For convenience, let us rather briefly recall how these $C^*$-algebras are defined.
Let $A=(a_{i, j})_{i, j\in I}$ be a matrix whose entries are indexed by a possibly infinite set
$I$. Suppose that  the entries of $A$
are all zeros and ones and that $A$ has no identically
zero rows. Let $X, Y$ be finite subsets of $I$. For every $j\in I$, define
$$a(X, Y, j)=\prod_{x\in X} a_{x, j}\prod_{y\in Y} (1- a_{y, j})\, .$$
Working under the additional condition that for any finite subsets $X, Y\subset I$ there are only finitely many
$j$’s such that $a(X, Y, j)=1$,  it is possible to define
a $C^*$-algebra $O_A$ as  the universal $C^*$-algebra
generated by a family of partial isometries $\{s_i: i\in I\}$, with initial projections $q_i:=s_i^*s_i$ and final projections $p_i=s_is_i^*$,  satisfying the following properties: 
\begin{enumerate}
\item $q_i$ and $q_j$ commute for all $i, j\in I$,
\item $p_i\perp p_j$, for all $i, j\in I$ with $i\neq j$
\item $q_i p_j = a_{i, j}p_j$, for all $i, j\in I$
\item $\prod_{x\in X} s_x^*s_x\prod_{y\in Y}(1-s_y^*s_y)=\sum_{j\in\bz} a(X, Y, j)s_js_j^*$\label{definingrel}\,.
\end{enumerate}
Notice that the above relations are exactly what in \cite{EL0} are referred to as $\textrm{TCK}_1)-\textrm{TCK}_3)$ and (1.3).
Furthermore,  condition (2) is equivalent to $s_i^*s_j=0$ if $i\neq j$ and
condition (3) is equivalent to $q_i s_j = a_{i, j}s_j$, for all $i, j\in I$.\\
Before continuing our analysis, it is worth pointing out that $O_A$ does not have a unit, as follows from
\cite[Proposition 8.5]{EL0}; we will denote the corresponding unitalized algebra by $\widetilde{O_A}$
(with unit $1$).\\
Henceforth, we will work with matrices indexed by $\bz$, the set  of integer numbers.
More precisely, we consider  $A=(a_{i, j})_{i, j\in\bz}$ with entries $a_{i,j}$  given by
$a_{i, j} = 1$ if $i\geq j$ and
$a_{i,j} = 0$ otherwise. Our next goal is to show that this particular
choice of the matrix leads to the sought abstract
characterization of the weakly monotone $C^*$-algebra.
\begin{thm}\label{charZ}
The map $O_A\ni s_i\mapsto A_i^\dag\in\mathcal{W}$, $i\in\bz$, extends to a
faithful representation of $O_A$. In particular, $\mathcal{W}\cong O_A$.
\end{thm}
\begin{proof}
In order to see that the map in the statement extends to a $*$-representation
of $O_A$, by universality of   $O_A$ it is enough   to ascertain that
the operators $A_i^\dagger$ satisfy the relations (1)--(4) (see also Theorem 8.6 in \cite{EL0}).
Now the first three relations are entirely obvious. Relation (4)
does need to be checked. To this end, let $X, Y$ be finite subsets of $\bz$.
We have to show that
there are only finitely many $j$’s such that $a(X, Y, j)$ is different from $0$
and that the equality
\begin{equation*}
\prod_{x\in X} A_xA^\dag_x\prod_{y\in Y}(1-A_yA^\dag_y)=\sum_{j\in\bz} a(X, Y, j)A_j^\dag A_j
\end{equation*}
holds.
We first observe that the left-hand side of the equality above is
$A_MA^\dag_M (1-A_NA^\dag_N)$, where
$M := \min X$ and $N := \max Y$. Note that
$a(X, Y, j)$ is different from $0$ if and only if $N <j \leq M$. There are two
cases to deal with, depending on whether $M > N$ or not.
If $M > N$, then the set of those
$j$ such that $a(X, Y, j)$ is different
from $0$ is non-empty, and the equality is satisfied as
$\sum_{N< j\leq M} A_j^\dag A_j=A_MA^\dag_M-A_NA^\dag_N$.
If $M \leq N$, then the set of the $j$’s such that $a(X, Y, j)$
is different from $0$ is empty and both sides of the equality are zero.\\
To conclude, we need to  prove that $\pi$ is a faithful
representation. To this end, we recall that it is possible to associate a graph with any matrix with either $0$ or
$1$ as its entries, see \cite{EL0}. The graph associated with our matrix $A$ is easily seen not to
have terminal circuits, see \cite{EL0} for the terminology.
An application of \cite[Therem 13.1]{EL0}  yields the following property enjoyed by  ideals of
$O_A$: any non-trivial two-sided ideal of $O_A$ must contain at least
one of the generators $s_i$. We are now ready to show the faithfulness of our representation.
We will argue by contradiction.
Let $\mathcal{I}\subset O_A$ be the
kernel of $\pi$. If $\mathcal{I}$ were not zero, then there should exist $i_0\in\bz$ such that
$s_{i_0}\in \mathcal{I}$. But then we  would have $A_{i_0}^\dag=\pi(s_ {i_0})=0$, which of course is not true.
\end{proof}

Our next aim is to show that  $O_A$ is not a type-I $C^*$-algebra. To accomplish this goal, we first show that
the  projection $P_\Om$ onto the vacuum does not sit in $O_A$ .

\begin{prop}\label{vacumnotalg}
The orthogonal projection $P_\Om$ onto the vacuum does not belong to $\mathcal{W}$.
\end{prop}

\begin{proof}
We shall argue by contradiction. Suppose on the contrary that $P_\Om$
does belong to $\mathcal{W}$. Then there exists a sequence $\{X_n: n\in\bn\}$, with
each $X_n$ sitting in the $*$-algebra $\mathcal{W}_0$ generated by $\{A_i : i\in\bz\}$ such that
$\|P_\Om - X_n\|\leq \frac{1}{2^n}$ for all $n$ in $\bn$.
There is no loss of generality if we assume that the $X_n$'s
are of the form $T_n +\sum_{k\in F_n}\a_k^{(n)} A_k A_k^\dag $, where, for each $n$,
$T_n$ is a finite linear combination of words in $\Lambda$ (as in Proposition \ref{Hamel}),
$F_n\subset \bz$ is a finite set, and $\alpha_k^{(n)}$ are complex coefficients for all $k$ in $F_n$.
Now  the inequality
\begin{equation}\label{norm}
\bigg\|  P_\Om - \big(T_n +\sum_{k\in F_n}\a_k^{(n)} A_k A_k^\dag \big) \bigg\|\leq\frac{1}{2^n}
\end{equation}
computed on the vacum vector $\Om$ gives
$ \big\|\Om - (T_n\Om +\sum_{k\in F_n}\a_k^{(n)} \Om)\big \|\leq\frac{1}{2^n}$. Because
$T_n\Om$ is either $0$ or a vector orthogonal to $\Om$, we find
$\big \|(1-\sum_{k\in F_n}\a_k^{(n)}) \Om\big \|\leq\frac{1}{2^n}$, that is
$\big|1-\sum_{k\in F_n}\a_k^{(n)}\big|\leq\frac{1}{2^n}$ for all $n$.\\
 In particular, we can fix $n=2$.
Computing Inequality \eqref{norm} on a vector $e_s$, where $s$ in $\bz$ is any integer smaller than the minimum
of the set of all indices of creators/annihilators showing up in $X_2$, we now find
$\big\|  T_2e_s -\sum_{k\in F_2} \a_k^{(2)} e_s\big \|\leq\frac{1}{4}$. As above, $T_2 e_s$ is orthogonal to $e_s$, hence the inequality implies that in particular we must also have $\big| \sum_{k\in F_2} \a_k^{(2)}\big|\leq\frac{1}{4}$, and an absurd has been got to since we also had
$\big|1-\sum_{k\in F_n}\a_k^{(2)}\big |\leq\frac{1}{4}$.
\end{proof}

\begin{cor}\label{nottypeI}
$O_A$ is not a $C^*$-algebra of type I.
\end{cor}

\begin{proof}
The Fock representation of $O_A$ is irreducible, but nevertheless  its range fails to contain
all compact operators because $P_\Om$ does not sit in $\mathcal{W}$. Thus the statement follows from a well-known characterization of separable type-I $C^*$-algebras \cite{G}.
\end{proof}

\begin{rem}
The unitalization $\widetilde{O_A}$  of $O_A$ continues not to be a type-I
$C^*$-algebra.
\end{rem}

\subsection{Ergodic properties of the shift}\label{erg}
Let us denote by $\tau$ the (unital) $*$-automorphism of $O_A$ ( or $\widetilde{O_A}$) uniquely determined by
$\tau(s_i)=s_{i+1}$ for all $i\in\bz$.
As we are interested in studying the ergodic properties of the shift, it might be worth recalling
the notions from ergodic theory that are relevant in our analysis. By a
$C^*$-dynamical system we mean a pair $(\ga, \Phi)$, where
$\ga$ is a (unital) $C^*$-algebra and $\Phi$ a (unital) $*$-homomorphism.
A state $\om$ on $\ga$ is invariant under $\Phi$ if $\om\circ\Phi=\om$.
The compact convex set of all invariant states of a $C^*$-dynamical system $(\ga, \Phi)$ is denoted by
$\mathcal{S}^\Phi(\ga)$.
A $C^*$-dynamical system (on a unital $C^*$-algebra) is said to be uniquely ergodic if $\mathcal{S}^\Phi(\ga)$ is a singleton, that is if the system only has one invariant state; in this case it is easy to see that the fixed-point subalgebra
$\ga^\Phi:=\{a\in\ga: \Phi(a)=a\}$ is trivial, {\it i.e.} $\ga^\Phi=\bc$.\\
More in general, a $C^*$-dynamical system is uniquely ergodic with respect to the fixed-point algebra if
any state of $\ga^\Phi$ has exactly one $\Phi$-invariant extension to the whole $\ga$. Unique ergodicity
with respect to the fixed-point subalgebra was first introduced in \cite{AD}, where it was characterized in terms
of a number of equivalent conditions: one of those is that for every $a$ in $\ga$ the Ces\`{a}ro average
$\frac{1}{n+1}\sum_{k=0}^{n+1}\Phi^k(a)$ converges in norm to an element of
$\ga^\Phi$.

\begin{prop}\label{notuniq}
The $C^*$-dynamical system $(\widetilde{O_A}, \tau)$ is not uniquely ergodic
w.r.t. the fixed-point subalgebra.
\end{prop}

\begin{proof}
Since $\widetilde{O_A}$ and $\widetilde{\mathcal{W}}$ are isomorphic, we may as well work at the Fock space level.
To prove the result, it is enough to exhibit an operator $T\in\widetilde{\mathcal{W}}$ such that the
Cesaro averages $\frac{1}{n}\sum_{k=0}^{n-1}\tau^{-k}(T)$ do not converge in norm.
We will show that $T=A_0A_0^\dagger$ will do.
First note that $\frac{1}{n}\sum_{k=0}^{n-1}\tau^{-k}(A_0A_0^\dagger)=\frac{1}{n}\sum_{k=0}^{n-1}A_{-k}A_{-k}^\dagger
$. Now the decreasing sequence of projections
$\{A_{-k}A_{-k}^\dagger: k\in\bn\}$ strongly converges to $P_\Om$. In particular, the sequence
$\frac{1}{n}\sum_{k=0}^{n-1}A_{-k}A_{-k}^\dagger$ also converges to $P_\Om$ strongly.
However, the convergence does not hold in norm because

$$\left\|P_\Om-\frac{1}{n}\sum_{k=0}^{n-1}A_{-k}A_{-k}^\dagger \right\|\geq
\left\|P_\Om e_{-n}-\frac{1}{n}\sum_{k=0}^{n-1}A_{-k}A_{-k}^\dagger e_{-n}\right\|=\|e_{-n}\|=1\, .$$

\end{proof}

\begin{lem}\label{estimate}
For any  word $Y\in\Lambda$, one has
$$\left\|\frac{1}{n}\sum_{k=0}^{n-1}\tau^k(Y)\right\|\leq\frac{1}{\sqrt{n}},\quad \textrm{for all}\,\, n\in\bn\, .$$
\end{lem}
\begin{proof}
We claim that the estimate
$$\left\|\sum_{j=1}^n A_j^\dag\eta_j\right\|^2\leq n\, \max_{1\leq j\leq n} \|\eta_j\|^2\,$$
 holds when the vectors $\eta_j$ all belong to a $k$-particle subspace, for some
fixed $k\geq 1$. The claim will be shown to hold at the end of the proof.\\
Now, let  $Y=(A^\dag_{i_1})^{k_1}\cdots (A^\dag_{i_m})^{k_m}A_{j_1}^{l_1}\cdots A_{j_n}^{l_n}$ be a word
featuring at least one creator. If
$\xi$ is a unit vector lying in a $m$-particle subspace,  define
$\xi_k:=\tau^k\big((A^\dag_{i_1})^{k_1- 1}(A^\dag_{i_2})^{k_2}\cdots (A^\dag_{i_m})^{k_m}A_{j_1}^{l_1}\cdots A_{j_n}^{l_n}\big)\xi$.
Note that $\|\xi_k\|\leq 1$ and that all vectors $\xi_k$ belong to a common $r$-particle space (with $r=m+k_1- 1+k_2+\ldots+ k_m-(l_1+\ldots +l_n)$).
But then we have
\begin{align*}
\left\|\sum_{k=0}^{n-1}\tau^k(Y)\xi\right \|= \left\|\sum_{k=0}^{n-1} A^\dag_{i_1 +k}\xi_k \right\|\leq\sqrt{n}
\end{align*}
where the last inequality is due to the fact that $\{A^\dag_{i_1 +k}\xi_k: k=0, \ldots, n-1\}$ is a set of orthogonal
vectors. The inequality we have arrived at is exactly the property in the statement as $m$ is arbitrary.\\
The case of a $Y$ with no creators at all is straightforwardly reconducted to the previous case by taking the adjoint of
$Y$. \\
All is left to do is prove the claim, which can be done as follows
\begin{align*}
&\left\|\sum_{j=1}^n A_j^\dag\eta_j\right\|^2=\left\langle \sum_{j=1}^n A_j^\dag\eta_j, \sum_{l=1}^n A_l^\dag\eta_l\right\rangle= \sum_{j, l=1}^n\left\langle  A_lA_j^\dag\eta_j, \eta_l\right\rangle=\\
& \sum_{j, l=1}^n \delta_{j, l}\left\langle  A_lA_j^\dag\eta_j, \eta_l\right\rangle= \sum_{j=1}^n
\|A_j^\dag\eta_j \|^2\leq \sum_{j=1}^n \|\eta_j \|^2\leq  n\max_{1\leq j\leq n} \|\eta_j\|^2.
\end{align*}
\end{proof}
The next result provides the description of all $\tau$-invariant states on $\widetilde{O_A}$. This set turns out to be the segment
whose endpoints are the vacuum state $\om_\Om$ (its natural extension to $\widetilde{O_A}$) on  $\widetilde{O_A}$ and the  state at infinity $\om_\infty$ (that is $\om_\infty(a+\lambda 1)=\lambda$, for all $a$ in $O_A$).

\begin{thm}\label{inv}
The set of all shift invariant states on $\widetilde{O_A}$ is given by
$$\mathcal{S}^\tau(\widetilde{O_A})=\{t\om_\Om+(1-t)\om_\infty: t\in [0,1]\}\,.$$
\end{thm}

\begin{proof}

Let $\om$ be a shift-invariant state on $\widetilde{O_A}\cong \widetilde{\mathcal{W}}$.
It suffices to show that the restriction of $\om$ to the dense $*$-algebra
$\mathcal{W}_0+\bc I\subset\widetilde{\mathcal{W}}$ is a convex combination
of the vacuum state and the state at infinity.
Now any element $X$ of $\mathcal{W}_0+\bc I$ can be written as a sum
$X=\gamma I+\sum_{\lambda\in F} c_\lambda Y_\lambda +\sum_{j\in G}\beta_j A_j A_j^\dag$, where
$F$ and $G\subset\bz$ are finite sets, for each $\lambda\in F$ the element $Y_\lambda$ is a word in the set
$\Gamma$ we defined above, and the $c_\lambda$'s and the $\beta_j$'s are complex coefficients.
By Lemma  \ref{estimate} we have that, for every
$\lambda$ in $F$, the averages  $\frac{1}{n}\sum_{k=0}^{n-1}\tau^k(Y_\lambda)$ converge in norm to $0$, which means
$\om(Y_\lambda)=\om\left(\frac{1}{n}\sum_{k=0}^{n-1}\tau^k(Y_\lambda)\right)$ must be zero as well.\\
Furthermore, by invariance $\om(A_j A_j^\dag)$ cannnot depend on $j\in\bz$, so we can define
$t:=\om(A_j A_j^\dag)$. Since $0\leq A_j A_j^\dag\leq I$, we have $0\leq t \leq 1$.
From the two equalities $\om_\infty(X)=\gamma$ and $\om_\Om(X)=\gamma+ \sum_j \beta_j$, we finally find
\begin{align*}
\om(X)&=\gamma+ t\sum_j \beta_j=\om_\infty(X)+ t(\om_\Om(X)-\om_\infty(X))\\
&=t\om_\Om(X)+(1-t)\om_\infty(X)\,  .
\end{align*}
\end{proof}
Our ultimate goal is to show that the fixed-point subalgebra of the shift is trivial. This will follow easily
from a possibly known general fact, which we include below for convenience.
\begin{lem}\label{extreme}
Let $(\ga, \Phi)$ be a $C^*$-dynamical system. Any pure state on $\ga^\Phi$ can be extended
to an extreme invariant state on $\ga$.
\end{lem}
\begin{proof}
Let $\om$ be a pure state on $\ga^\Phi$. Define the set
$$C_\om:=\{\varphi\in \mathcal{S}^\Phi(\ga): \varphi\upharpoonright_{\ga^\Phi}=\om \}\, .$$
Observe that $C_\om$ is a non-empty (weakly*) compact convex set.
Therefore, by the Krein-Milman theorem $C_\om$ contains extreme points.
Let $\varphi_0$ any such state. We need to show that $\varphi_0$ is extreme in
$\mathcal{S}^\Phi(\ga)$. To this end, let $\varphi_1, \varphi_2$ be states in $\mathcal{S}^\Phi(\ga)$
such that $\varphi_0=t\varphi_1+(1-t)\varphi_2$ for some $t$ with $0<t<1$.
In particular, by restricting the above equality to the fixed-point subalgebra we find
$$\om=\varphi_0\upharpoonright_{\ga^\Phi}=t\varphi_1\upharpoonright_{\ga^\Phi}+(1-t)\varphi_2\upharpoonright_{\ga^\Phi}\,.$$
Since $\om$ is pure by assumption, we must have
$\varphi_1\upharpoonright_{\ga^\Phi}=\varphi_2\upharpoonright_{\ga^\Phi}=\om$.
This means that $\varphi_1$ and $\varphi_2$ actually lie in $C_\om$.
But because $\varphi_0$ is extreme in $C_\om$, $\varphi_1$ and $\varphi_2$ must be the same state, and the proof
is complete.

\end{proof}

\begin{thm}\label{trivfix}
The fixed-point subalgebra of the dynamical system $(\widetilde{O_A}, \tau)$ is trivial.
\end{thm}

\begin{proof}

Let $T\colon \mathcal{S}^\tau(\widetilde{O_A})\rightarrow\mathcal{S}({\widetilde{O_A}^\tau})$ be the restriction map, that is
$$T(\varphi):=\varphi\upharpoonright_{O_A^\tau}\,, \,\,\, \varphi\in  \mathcal{S}^\tau(\widetilde{O_A})\, .$$
 $T$ is a surjective continuous affine map between compact convex sets by Lemma \ref{extreme}.
Now, by Theorem \ref{inv} the convex set $\mathcal{S}^\tau(\widetilde{O_A})$ has exactly two extreme points as it is a segment.
Therefore, the fixed-point subalgebra $\widetilde{O_A}^\tau$ can have at most two pure states as follows from Lemma \ref{extreme}.
The conclusion will then be achieved if we show that $\widetilde{O_A}^\tau$ cannot have exactly two pure states.
If  $\widetilde{O_A}^\tau$ had two pure states, then $\mathcal{S}(\widetilde{O_A}^\tau)$ would be a segment as well and
the map $T$ would thus be a homeomorphism.
However, this is  not  the case since $T$  cannot be injective as $(\widetilde{O_A}, \tau)$ is not uniquely ergodic w.r.t.\ the fixed-point
subalgebra, as shown in Proposition \ref{notuniq}.
The only possibility for $\widetilde{O_A}^\tau$ is then to have one pure state, which is the same as having $\widetilde{O_A}^\tau=\mathbb{C}1$.
\end{proof}

\begin{rem}
Without any further effort, one can also see that
$O_A^\tau=0$ and $\mathcal{S}^\tau(O_A)=\{\om_\Om\}$.
\end{rem}

\subsection{Irreducibility of the self-adjoint subalgebra }
For every $i\in\bz$, we define a weakly monotone position operator as
$X_i:= A_i + A_i^\dag$. Throughout the paper the unital $C^*$-algebra generated by all position operators,
$\mathcal{A}:= C^*(I,\{X_i: i\in\bz\})$, will be referred to as the self-adjoint subalgebra of
$\mathcal{W}$.

\begin{lem}\label{limit}
The limit equality
$$\lim_{N\rightarrow\infty}\frac{1}{2N+1}\sum_{i=-N}^N X_i^2=P_\Om+\frac{1}{2}(I-P_\Om)$$
holds in the strong operator topology.
\end{lem}

\begin{proof}

The proof follows the same idea as  \cite[Lemma 2.2]{W06}.
Because the sequence $\left\{\frac{1}{2N+1}\sum_{i=-N}^N X_i^2\right\}_{n\in\bn}$ is norm bounded, it is enough
to ascertain the stated equality only on vectors of the canonical orthonormal basis of the weakly monotone Fock space.\\
To begin with, note that for every $i\in\bz$ one has $X_i^2\Om= \Om+ e_i\otimes e_i$, hence
$$\frac{1}{2N+1}\sum_{i=-N}^N X_i^2\Om= \Om+\frac{1}{2N+1}\sum_{i=-N}^N e_i\otimes e_i\, .$$\\
Because $\|\frac{1}{2N+1}\sum_{i=-N}^N e_i\otimes e_i\|^2=\frac{1}{2N+1}$, the sequence
$\frac{1}{2N+1}\sum_{i=-N}^N X_i^2\Om$ converges in norm to $\Om$.\\

Let now $e_{i_1}\otimes\cdots\otimes e_{i_k}$ be any basis vector orthogonal to the vacuum vector, {\it i.e.}
$k\geq 1$. Note that for every $i> i_1$ one has
\begin{equation}\label{Xsquare}
X_i^2 (e_{i_1}\otimes\cdots\otimes e_{i_k}) =e_{i_1}\otimes\cdots\otimes e_{i_k}+ e_i\otimes e_i\otimes e_{i_1}\cdots\otimes e_{i_k}
\end{equation} whereas for $i<i_1$ one has
$$X_i^2 (e_{i_1}\otimes\cdots\otimes e_{i_k}) =0\, .$$
Now we have
\begin{align*}
&\frac{1}{2N+1}\sum_{i=-N}^N X_i^2 e_{i_1}\otimes\cdots\otimes e_{i_k}=\frac{1}{2N+1}\sum_{i=i_1}^N X_i^2 e_{i_1}\otimes\cdots\otimes e_{i_k}\\
&=\frac{1}{2N+1} X_{i_1}^2e_{i_1}\otimes\cdots\otimes e_{i_k}+\frac{1}{2N+1}\sum_{i=i_1+1}^N X_i^2 e_{i_1}\otimes\cdots\otimes e_{i_k},
\end{align*}
thus $$\lim_{N\rightarrow\infty}\frac{1}{2N+1}\sum_{i=-N}^N X_i^2 e_{i_1}\otimes\cdots\otimes e_{i_k}=\lim_{N\rightarrow\infty}\frac{1}{2N+1}\sum_{i=i_1+1}^N X_i^2  e_{i_1}\otimes\cdots\otimes e_{i_k}\, .$$
By taking \eqref{Xsquare} into account, the term in the right-hand side of the above equality  can be written as
\begin{align*}
\frac{N-i_1}{2N+1}e_{i_1}\cdots\otimes e_{i_k}+\frac{1}{2N+1}\sum_{i=i_1+1}^N e_i\otimes e_i\otimes e_{i_1}\cdots\otimes e_{i_k}\,.
\end{align*}
Now $$\lim_{N\rightarrow\infty} \frac{N-i_1}{2N+1}e_{i_1}\cdots\otimes e_{i_k}=\frac{1}{2} e_{i_1}\cdots\otimes e_{i_k}$$
and
$$\lim_{N\rightarrow\infty}\frac{1}{2N+1}\sum_{i=i_1+1}^N e_i\otimes e_i\otimes e_{i_1}\cdots\otimes e_{i_k}=0$$
since $\left\|\sum_{i=i_1+1}^N e_i\otimes e_i\otimes e_{i_1}\cdots\otimes e_{i_k}\right\|=\sqrt{N- i_1}$ by orthogonality.
Putting everything together, we finally have
$$\lim_{N\rightarrow\infty}\frac{1}{2N+1}\sum_{i=-N}^N X_i^2 e_{i_1}\otimes\cdots\otimes e_{i_k}=\frac{1}{2}e_{i_1}\otimes\cdots\otimes e_{i_k}$$
and the proof is complete.
\end{proof}

\begin{prop}\label{selfad}
The self-adjoint subalgebra $\mathcal{A}\subset\mathcal{W}$ is an irreducible subalgebra of $\mathcal{B}(\mathcal{F}_{wm}(\ch))$.
\end{prop}

\begin{proof}
We will show that the commutant $\mathcal{A}'$ equals $\mathbb{C}I$.
To this aim, we first point out that $P_\Om$ lies in $\mathcal{A}''$. Indeed, thanks to
Lemma \ref{limit}, we know that $T:=P_\Om+\frac{1}{2}(I-P_\Om)$ certainly sits in $\mathcal{A}''$, which means
$P_\Om$ does the same being a spectral projection of $T$.\\
The second step to take is to prove that $\Om$ is cyclic for $\mathcal{A}$.
We claim that for every $n\in\bn$ and $i\in \bz$ there exists a polynomial $q_n$ (which only depends on $n$) such that
$$q_n(X_i)\Om=e_i^{\otimes^n}\, .$$
This can be seen by induction on $n$. For $n=0$ or $n=1$, there is nothing to prove as $q_0(x)=1$ and $q_1(x)=x$ will obviously do.
As for the inductive step, observe that from $q_n(X_i)\Om=e_i^{\otimes^n}$ we find
$$X_iq_n(X_i)\Om=e_i^{\otimes^{n-1}}+ e_i^{\otimes^{n+1}}= q_{n-1}(X_i)\Om+e_i^{\otimes^{n+1}}$$
which says that $q_{n+1}(x):=xq_n(x)-q_{n-1}(x)$ satisfies $q_{n+1}(X_i)\Om=e_i^{\otimes^{n+1}}$.\\
For a general  basis vector, say $e_{i_1}^{n_1}\otimes\cdots\otimes e_{i_k}^{n_k}$ with
$i_1>\ldots>i_k$, it is easy to see that
$$q_{n_1}(X_{i_1})\cdots q_{n_k}(X_{i_k})\Om=e_{i_1}^{n_1}\otimes\cdots\otimes e_{i_k}^{n_k}\, .$$
We are now in a position to reach the conclusion. If $E$ is a projection in $\mathcal{A}'$, we must have
$EP_\Om=P_\Om E$ from which we see $E\Om=\alpha\Om$ with $\alpha$ being either $0$ or $1$.
But then we have $ET\Om=TE\Om=\alpha T\Om$ for every $T\in\mathcal{A}$.  By cyclicity of $\Om$ we finally have $E=\alpha I$.
\end{proof}

\section{Weakly-monotone $C^*$-algebra over $\bn$}

As with $\bz$, a weakly monotone Fock Hilbert space can be associated with $\bn$ as well:
$$\mathcal{F}_{wm}(\ell^2(\bn)):=\bigoplus_{k=0}^\infty H_k\,$$
where $H_k$ is the linear span of the set of vectors $e_{i_1}\otimes e_{i_2}\otimes\cdots\otimes e_{i_k}$ with
$k\in\bn$   and $ i_1\geq i_2\geq\cdots\geq i_k\geq 1$, whereas $H_0=\bc\Om$.
On $\mathcal{F}_{wm}(\ell^2(\bn))$ weakly monotone creation (annihilation) operators
$A_i^\dag$ ($A_i$), $i$ in $\bn$, can be defined in the exact same way as we did for $\bz$.
We denote by $\mathcal{W}_\bn\subset \cb(\mathcal{F}_{wm}(\ell^2(\bn)))$ the $C^*$-subalgebra generated
by the set $\{A_i, A_i^\dag: i\in\bn\}$.
Unlike the case of $\bz$, the projection $P_\Om$ onto the Fock vacuum does sit in $\mathcal{W}_\bn$, because of the identity
$$P_\Om= A_1A_1^\dag- A_1^\dag A_1\, .$$
As was the case with $\mathcal{W}$,  it is not too difficult to show that $I$ does not belong to $\mathcal{W}_\bn$ either.
We denote by $O_{A_\bn}$ the universal $C^*$-algebra generated by a countable family
$\{s_i: i=0, 1, \ldots\}$ of partial isometries satisfying the following relations:
\begin{equation}\label{defrelN1}
s_i^*s_j=0\,, \quad \textrm{for all}\, \quad i\neq j
\end{equation}
\begin{equation}\label{defrelN2}
s_i^* s_i =\sum_{k=0}^i s_ks_k^*\,, \quad\textrm{for all}\,\, i=0, 1, \ldots
\end{equation}

Note that $O_{A_\bn}$ is a well-defined $C^*$-algebra, because the maximal seminorm on the universal
$*$-algebra generated by $\{s_i: i=0, 1, \ldots\}$ is finite. In addition,  our $C^*$-algebra $O_{A_\bn}$ is not trivial as
$\mathcal{W}_\bn$ is clearly a representation of  $O_{A_\bn}$.\\
Also note that for any $z$ in $\bt$ the set $\{zs_i: i=0, 1, \ldots\}$ still satisfies the defining relations of $O_{A_\bn}$, which by
universality implies that there exists an automorphism $\a_z\in{\rm Aut}(O_{A_\bn})$ uniquely determined by
$\a_z(s_i)= zs_i$ for $i\geq 0$. As is done in the literature, we will refer to the automorphisms
$\{\a_z: z\in\bt\}$ as the gauge automorphisms  or as the gauge action of $\bt$ on $O_{A_\bn}$.\\

As a consequence of relations \eqref{defrelN1}--\eqref{defrelN2}, the generators of $O_{A_\bn}$ also satisfy
relations that are more reminiscent of the concrete weakly monotone $C^*$-algebra. To write these relations down, for $i, j$ natural numbers we define $a_{i, j}$ to be $1$ if $i\geq j$ and $0$ otherwise.
\begin{lem}\label{extrarel}
The generators of  $O_{A_\bn}$ also satisfy the following relations:\\
\begin{itemize}
\item [(i)] $s_is_j=0$ if  $i< j$
\item [(ii)] $s_i^*s_is_j=a_{i, j}s_j$
\end{itemize}
\end{lem}
\begin{proof}
As for the first equality, if $i< j$  we have
$s_is_j= s_is_i^*s_is_j= \sum_{k=0}^i s_i s_ks_k^* s_j=0$ thanks to
\eqref{defrelN1} and \eqref{defrelN2}.\\
The second equality is got to similarly. Since
$s_i^*s_is_j= \sum_{k=0}^i s_ks_k^* s_j$, for $s_i^*s_is_j$ not to vanish we must have
$i\geq j$, in which case only the $j$-th summand survives yielding the claimed equality.
\end{proof}
We would like to point out that $O_{A_\bn}$ can also be presented as
the Exel-Laca $C^*$-algebra \cite {EL0} associated with the infinite matrix
$A_\bn:= (a_{i, j}: i, j=0, 1, \ldots)$ with $a_{i, j}=1$ for all $i, j$ with $i\geq j$ and $0$ otherwise.
This makes it possible to apply \cite[Proposition 8.5]{EL0} once again to see that $O_{A_\bn}$ fails
to be unital.\\
We set $\mathcal{A}_n:=C^*(s_0, \ldots s_n)\subset O_{A_\bn}$.  Note that $\mathcal{A}_n$ is invariant
for the gauge automorphisms, which clearly act on  $\mathcal{A}_n$ as automorphisms.\\
Denote by $A_n$ the $n+1$ by $n+1$ matrix given by  $A_n:=(a_{i, j})_{i, j=0, \ldots n}$, where
$a_{i, j}$ is the symbol we introduced before stating Lemma \ref{extrarel}.
Let $\mathcal{O}_{A_n}$ denote
the Cuntz-Krieger algebra associated with the matrix $A_n$. As is done in \cite{CK}, we denote by
$\{S_i: i=0, 1, \ldots, n\}$ the generators of $\mathcal{O}_{A_n}$. Note that, for each $n\geq 0$,
$\mathcal{O}_{A_n}$ is a unital $C^*$-algebra with unit $S_n^*S_n$.

\begin{lem}
$\mathcal{A}_n$ is canonically isomorphic with the Cuntz-Krieger algebra $\mathcal{O}_{A_n}$.
\end{lem}

\begin{proof}
By universality of the Cuntz-Krieger algebra $\mathcal{O}_{A_n}$, there must be an epimorphism
$\Psi\colon \mathcal{O}_{A_n}\rightarrow \mathcal{A}_n$ such that $\Psi(S_i)=s_i$ for all $i=0, 1, \ldots, n$.
Finally, the epimorphism $\Psi$ is injective as well as follows from Theorem
2.3 in \cite{AHR}, which applies since $\mathcal{A}_n$ is acted upon by the gauge automorphisms.
\end{proof}
By its very definition, $O_{A_\bn}$ can be seen as the inductive limit of the inductive system $\{\mathcal{A}_n: n=1, \ldots\}$. Note that, however, the natural inclusions $\mathcal{A}_n\subset \mathcal{A}_{n+1}$
fail to be unital.\\

For every $z\in\bt$, we consider the representation $\widetilde{\pi}_{0, z}\colon O_{A_\bn}\rightarrow\cb(\mathcal{F}_{wm}(\ell^2(\bn))) $
given by $\widetilde{{\pi}}_{0, z}(s_0):= z P_\Om$ and $\widetilde{{\pi}}_{0, z}(s_i):=A_i^\dag$ for all $i\in\bn$.
\begin{rem}
It is worth noting that if a representation $\pi\colon O_{A_\bn}\rightarrow\cb(\mathcal{F}_{wm}(\ell^2(\bn)))  $ satisfies $\pi(s_i)=A_i^\dag$ for all $i\geq 1$, then $\pi(s_0)= zP_\Om$ for some $z$ in $\bt$. This follows from the equality $P_\Om= A_1A_1^\dag- A_1^\dag A_1$. Indeed, we have
$\pi(s_0^*s_0)= \pi(s_1^*s_1- s_1s_1^*)= P_\Om$, hence
$\pi(s_0)=\pi(s_0 s_0s_0^*)= \pi(s_0)P_\Om$, and thus  the only possibility is that
$\pi(s_0)=zP_\Om$ for some $z$ in $\bt$ because $\pi(s_0)$ is a partial isometry.
\end{rem}
Choosing $z=1$ yields the so-called Fock representation, which we will denote by $\pi_F$ (rather than $\widetilde{{\pi}}_{0, 1}$).

\begin{prop}\label{faithfulness}
The direct sum $\pi:=\oplus_{z\in\bt}\,\,\widetilde{\pi}_{0, z}$ is a faithful representation of $O_{A_\bn}$.
\end{prop}

\begin{proof}
By applying  Theorem 2.3 in \cite{AHR} we see that, for each $n\in\bn$, the direct sum $\pi=\oplus_{z\in\bt}\,\,\widetilde{\pi}_{0, z}$ restricts to a faithful representation of
$\mathcal{A}_n$ because the gauge automorphisms are unitarily implemented on the corresponding Hilbert space. Indeed, if for any $w\in\mathbb{T}$, we define the unitaries $V_w$ acting on $\mathcal{F}_{wm}(\ell^2(\bn))$  and $U_w$ acting on $\bigoplus_{z\in\mathbb{T}}\mathcal{F}_{wm}(\ell^2(\bn))$ by
$$V_w(\bigoplus_{k\geq0}x_k):=\bigoplus_{k\geq0}w^k x_k\,,\quad x_k\in H_k$$
$$U_w:=\bigoplus_{z\in\mathbb{T}}V_w\, ,$$
then it is very easy to check that $U_w\pi(x)U_w^*=\pi(\alpha_w(x))$ for all $x\in O_{A_\bn}, w\in\mathbb{T}$.
In particular, the restriction of $\pi$ to $\mathcal{A}_n$ is isometric, thus $\pi$ is isometric on $\bigcup_n\mathcal{A}_n$, which proves the claim as $\bigcup_n\mathcal{A}_n$ is dense in $O_{A_\bn}$.
\end{proof}
\begin{rem}
The Hilbert space which the representation $\pi$ above acts on is clearly non-separable.
\end{rem}
The proposition above can also be obtained by directly applying Theorem 2.7 in  \cite{RS} if one regards
$O_{A_\bn}$ as an Exel-Laca algebra.\\

Let $J\subset O_{A_\bn}$ be the ideal generated by $s_0 -s_0s_0^*$.

\begin{prop}\label{faithfulrep}
The quotient $C^*$-algebra $O_{A_\bn}/J$ is isomorphic with $\mathcal{W}_\bn$
through the Fock representation $\pi_F$ of the former.
\end{prop}
\begin{proof}
We denote by $\mathcal{B}\subset O_{A_\bn}$ the $C^*$-subalgebra generated by
the set $\{s_0s_0^*, s_1, s_2, \ldots\}$. The restriction of $\pi$ to $\mathcal{B}$ is faithful and
is a direct sum of Fock representations with the same phase $z=1$. Therefore, $\pi_F$ is a faithful representation
of $\mathcal{B}$. Since $\pi_F(\mathcal{B})=\mathcal{W}_\bn$, the statement will be proved if we show that
$O_{A_\bn}/J$ and $\mathcal{B}$ are isomorphic. To this end, note that by universality $O_{A_\bn}$ projects onto
$\mathcal{B}$, which yields an epimorphism $\Psi:O_{A_\bn}/J \rightarrow \mathcal{B}$ given by
$\Psi([s_0]_J)= s_0s_0^*$ and, for $i\geq 1$,  $\Psi([s_i]_J)= s_i$. Injectivity of $\Psi$ can be proved by exhibiting
its inverse $\Phi:\mathcal{B}\rightarrow O_{A_\bn}/J$, which is obtained by composing the inclusion of
$\mathcal{B}$ in $O_{A_\bn}$ with the canonical projection of the latter onto $O_{A_\bn}/J$. Explicitly, $\Phi$ acts on the set of generators as $\Phi(s_0s_0^*)= [s_0]_J$ and, for $i\geq 1$, $\Phi(s_i)= [s_i]_J$.
\end{proof}
As a straightforward application of Proposition \ref{faithfulrep}, we derive the following space-free characterization of the weakly
monotone $C^*$-algebra.
\begin{cor}\label{weaklyrel}
The weakly monotone $C^*$-algebra over $\bn$ is isomorphic to the universal $C^*$-algebra generated by a countable family of partial isometries $\{s_i: i=0, 1, \ldots\}$ satisfying \eqref{defrelN1}--\eqref{defrelN2} and $s_0=s_0s_0^*$.
\end{cor}

For every $n\geq 1$, we denote by $J_n\subset O_{A_\bn}$ the closed two-sided ideal generated by the finite set $\{s_0, s_1, \ldots, s_{n-1}\}$ and by $p_n$ the canonical projection onto the quotient $p_n:O_{A_\bn}\rightarrow O_{A_\bn}/J_n$. That said, we can end this section with a proposition which will come in useful
when we deal with the representation theory of  $O_{A_\bn}$.

\begin{prop}\label{iso}
For every $n\geq 1$, the quotient $O_{A_\bn}/J_n$ is isomorphic to $O_{A_\bn}$ through the isomorphism
$\Psi_n:O_{A_\bn}\rightarrow O_{A_\bn}/J_n$ given by
$\Psi_n(s_i)=[s_{i+n}]_{J_n}$ for every $i\geq 0$.
\end{prop}

\begin{proof}
We start by noting that the map $\Psi_n$ in the statement is a well-defined $*$-homomorphism by
universality of $O_{A_\bn}$ because the set $\{[s_{i+n}]_{J_n}: i\geq 0\}$ satisfies
\eqref{defrelN1} and \eqref{defrelN2}.\\
In order to prove that $\Psi_n$ is injective as well, we observe that the gauge action of $\bt$
passes to the quotient algebra $O_{A_\bn}/J_n$ since $J_n$ is gauge invariant. Reasoning exactly as we
did in the proof of Proposition \ref{faithfulness}, we can conclude that $\Psi_n$ is a $*$-isomorphism.
\end{proof}

\subsection{Representation theory}

The next result, which can be seen as a generalization of Theorem 2.7 in \cite{AHR},
characterizes $\widetilde{{\pi}}_{0, z}$ as the sole irreducible representations of
$O_{A_\bn}$ sending $s_0$ to a non-zero operator.

\begin{prop}\label{int}
If $\pi$ is an irreducible representation of $O_{A_\bn}$ such that $\pi(s_0)\neq 0$, then there exists
$z$ in $\bt$ such that $\pi\cong\widetilde{{\pi}}_{0, z}$.

\end{prop}

\begin{proof}
Throughout the proof, $\ch_\pi$ will denote the Hilbert space of the representation $\pi$.
Note that $\ch_\pi$ is a separable Hilbert space thanks to irreducibility of $\pi$ and separability of $O_{A_\bn}$.
Set $U:=\pi(s_0)$. Since $U$ is a normal partial isometry, the initial and final space of $U$ coincide.
Denote this common (closed) subspace of  $\ch_\pi$ by $\ck$, and observe that
the restriction of $U$ to $\mathcal{K}$ is a unitary, which we continue to denote by $U$.
We will show that irreducibility of $\pi$ implies that $\ck$ must be one-dimensional.\\
More precisely, we next show that if $\ck$ fails to be one-dimensional,  the non-trivial decomposition
of $U$ as a direct integral results in a corresponding decomposition of our representation
$\pi$.\\
Since $\ch_\pi$
is separable, $U$ can be decomposed into a direct integral as follows.
There exist a measurable field of (non-zero) Hilbert spaces
$\s(U)\ni z\mapsto \ck(z)$ and a Borel probability measure
$\mu$ on $\s(U)$ such that $\ck=\int^\oplus_{\s(U)} \ck(z){\rm d}\mu$ and
$U=\int^\oplus_{\s(U)} z1_{\ck(z)}\di\mu$ (up to unitary equivalence), see
{\it e.g.} Theorem 1. in \cite[Chapter 6]{Dix}.
We also recall that
$\mu$ is a basic measure for $U$, so its support is the whole $\s(U)$.\\
Consider the constant field of Hilbert spaces
$\s(U)\ni z\mapsto \ch(z)$, where each
$\ch(z)$ is the weakly monotone Fock space $\mathcal{F}_{wm}(\ell^2(\bn))$.
Note that $\s(U)\ni z\mapsto\widetilde{\pi}_{0, z}$ is a measurable field of representations
of $O_{A_\bn}$ because $\widetilde{\pi}_{0,z}(s_i)$ does not depend on $i$ unless $i=0$, in which
case  $\widetilde{\pi}_{0,z}(s_0)=zP_\Om$ that is clearly a measurable field of operators.\\
We next aim to show that the direct integral representation
$\int^\oplus_{\s(U)} \widetilde{\pi}_{0,z}{\rm d}\mu$ can be realized as a subrepresentation of $\pi$.
This will be done by exhibiting an isometry $W$ that intertwines $\int^\oplus_{\s(U)} \widetilde{\pi}_{0,z}{\rm d}\mu$
with $\pi$. To this end, we first need to set up some notation.
We denote by $\Gamma$ the set of all ordered $k$-tuples $\gamma=(i_1, \ldots, i_k)$, where $k, i_1, \ldots, i_k$ are in $\bn$
and $i_1\geq\ldots\geq i_k\geq 1$. We also allow $k$ to be $0$ , in which case $\gamma=\emptyset$.
For any $\gamma=(i_1, \ldots, i_k)\in\Gamma$, we set $e_\gamma=e_{i_1}\otimes\cdots\otimes e_{i_k}$ (with
$e_\emptyset= \Om$) as vectors in $\mathcal{F}_{wm}(\ell^2(\bn))$, and $s_\gamma= s_{i_1}\cdots s_{i_k}$ as elements
of the abstract algebra $O_{A_\bn}$ (assuming $s_\emptyset=I$ by convention).
Note that $\overline{{\rm span}}\{f e_\gamma: \gamma\in \Gamma, f\in L^2(\s(U), \mu)\}
=\int^\oplus_{\s(U)} \ch(z){\rm d}\mu$, and $f e_\gamma\perp g e_\beta$ if $\gamma\neq \beta$, because
$\int^\oplus_{\s(U)} \ch(z){\rm d}\mu$ is isometrically isomorphic with the Hilbert tensor product $L^2(\s(U), \mu)\overline{\otimes}\mathcal{F}_{wm}(\ell^2(\bn))$.\\
Finally, we need to consider a measurable field $\s(U)\ni z\mapsto \Om(z)\in\ck(z)$ of unit vectors, {\it i.e.}
$\|\Om(z)\|^2=1$ for all $z\in\s(U)$. This allows us to embed $L^2(\sigma(U), \mu)$ into $\mathcal{K}$
through the isometry $V: L^2(\sigma(U), \mu)\rightarrow\mathcal{K}$, which takes any
function $f$ in $L^2(\sigma(U), \mu)$ to the measurable field $(Vf) (z):= f(z)\Om(z)$,  {\rm a.e.} for $z\in\s(U)$.\\
Having established the necessary notation, we are now ready to define
$W: \int^\oplus_{\s(U)} \ch(z){\rm d}\mu\rightarrow\ch_\pi $. We start by defining $W$ on a dense subspace as
$$ W\,\, \sum_{\gamma\in \Gamma_0} f_\gamma e_\gamma:=\sum_{\gamma\in\Gamma_0}\pi(s_\gamma)Vf_\gamma\,,$$
where $\Gamma_0$ is a finite subset of $\Gamma$. The squared norm of the vector in the left-hand side of the definition
above can be computed easily as follows:
\begin{align*}
\big\|\sum_{\gamma\in \Gamma_0} f_\gamma e_\gamma\big\|^2=\sum_{\gamma\in \Gamma_0} \int_{\s(U)}|f_\gamma|^2{\rm d}\mu\, .
\end{align*}
The squared norm of the vector in the right-hand side is a shade more laborious to compute. To this end, for every
$\gamma$ in $\Gamma_0$ define
$\xi_\gamma:= Vf_\gamma$. By definition, the vectors
$\xi_\gamma$ lie in $\ck$ (that is in ${\rm Ran}\, U$) for all $\gamma$ in $\Gamma_0$. Since
the initial projection of $\pi(s_\gamma)$ contains $\ck$ for all $\gamma$, we have
\begin{align*}
\|\pi(s_\gamma)\xi_\gamma\|^2= \|\xi_\gamma\|^2= \|f_\gamma\|^2 \int_{\s(U)} |f_\gamma|^2{\rm d}\mu\, .
\end{align*}
Now for different $\gamma, \gamma'$ in $\Gamma_0$ the vectors $\pi(s_\gamma)\xi_\gamma$ and $\pi(s_{\gamma'})\xi_\gamma'$  are seen at once to be orthogonal, which means
\begin{align*}
\|\sum_{\gamma\in \Gamma_0} \pi(s_\gamma)\xi_\gamma\|^2=\sum_{\gamma\in\Gamma_0} \|\xi_\gamma\|^2=\sum_{\gamma\in\Gamma_0}  \int_{\s(U)} |f_\gamma|^2{\rm d}\mu\,,
\end{align*}
which shows that $W$ can be extended to an isometry defined on the whole $\int^\oplus_{\s(U)} \ch(z){\rm d}\mu$.\\
We now move on to show that $W$ is an intertwiner. For this, it is sufficient to check the involved equality only on the generators of
$O_{A_\bn}$, namely that $W \int^\oplus_{\s(U)} \widetilde{\pi}_{0,z}(s_i){\rm d}\mu= \pi(s_i) W$ for all
$i\geq 0$. By linearity and boundedness of the involved operators it is enough to ascertain our equality only on vectors of the type
$fe_\gamma$, with $f$ in $L^2(\s(U), \mu)$ and $\gamma=(i_1, \ldots, i_k)$ in $\Gamma$. This is a matter of easy computations. Indeed, for $i\geq 1$ we have:

\begin{align*}
W \bigg(\int^\oplus_{\s(U)} \widetilde{\pi}_{0,z}(s_i){\rm d}\mu\bigg)\, fe_\gamma&=
 W f_\gamma \widetilde{\pi_z}(s_i) e_\gamma\\
&= \pi(s_{(i, \gamma)})Vf_\gamma\\
&=\pi(s_i)\pi(s_\gamma)Vf_\gamma\\
&=\pi(s_i)W\,f e_\gamma\,,
\end{align*}
where $s_{(i, \gamma)}$ is $0$ if $i<i_1$ and $s_{(i, \gamma)}= s_{(i, i_1, \cdots, i_k)}$ if $i\geq i_1$.\\
The case $i=0$ is dealt with separately. First, if $i=0$ and $\gamma$ is not empty, then the equality
is trivially satisfied because both sides vanish. If $i=0$ and $\gamma=\emptyset$, then we have:
\begin{align*}
W \bigg(\int^\oplus_{\s(U)} \widetilde{\pi}_{0,z}(s_0){\rm d}\mu\bigg)\, f\Om=
 Wg\Om= Vg= U Vf=\pi(s_0) W (f\Om)\,,
\end{align*}
where $g$ in $L^2(\s(U), \mu)$ is defined as $g(z):=zf(z)$ ($\mu$ {\rm a.e.}).\\
Now the direct integral $\int^\oplus_{\s(U)} \widetilde{\pi}_{0,z}{\rm d}\mu(z)$ cannot  be irreducible unless
$\sigma(U)$ is a singleton, say $\s(U)=\{z\}$ and $\mu$ is the Dirac measure $\delta_z$ and
$K(z)=\bc$.\\

\noindent
From now on we can then assume $\ck=\bc\xi$, where $\xi$ is a norm-one vector. Since $U$ is unitary on $\ck$, there exists
$z$ in $\bc$ with $|z|=1$ such that $U\xi=z\xi$. Thanks to what we saw above, $\widetilde{{\pi}}_{0, z}$ can be realized as a subrepresentation of $\pi$. But because $\pi$ is irreducible,  we must have $\widetilde{{\pi}}_{0, z}\cong \pi$.
\end{proof}

As a straightforward consequence of the above result, we can derive the following
characterization of the Fock representation of $\mathcal{W}_\bn$.
\begin{cor}
The Fock representation $\pi_F$ is the only irreducible representation of the weakly monotone $C^*$-algebra that sends
$s_0$ to a non-zero projection.
\end{cor}

That said, we continue our analysis of irreducible representations of $O_{A_\bn}$.
For every $n\geq 1$ and $z$ in $\bt$, we can consider the representations of $O_{A_\bn}$
defined as $\widetilde{\pi}_{n,z}:=\widetilde{\pi}_{0, z}\circ\Psi_n^{-1}\circ p_n$, which are irreducible by definition. More importantly, they exhaust the catalogue of all irreducible representations of $O_{A_\bn}$.

\begin{thm}\label{catalogue}
If $\pi$ is an irreducible representation of $O_{A_\bn}$, then there exists $n\geq 0$ such that $\pi\cong\widetilde{\pi}_{n,z}$ for some $z$ in $\bt$.
\end{thm}

\begin{proof}
Let $n$ be the minimum of the set $\{i\in\bn: \pi(s_i)\neq 0\}$.
In other words, our representation satisfies $\pi(s_i)=0$ for all $i=0, 1, \ldots, n-1$ and
$\pi(s_i)\neq 0$ for all $i\geq n$. Since ${\rm Ker}\, \pi$ contains $J_n$, we see that
$\pi$ factors as $\pi=\rho\circ p_n$, where $\rho$ is an irreducible representation of
the quotient algebra $O_{A_\bn}/J_n$. By Proposition \ref{iso} the quotient algebra is isomorphic with
$O_{A_\bn}$ itself through the isomorphism $\Psi_n$ given by $\Psi_n(s_i)=[s_{i+n}]_{J_n}$, for
$i\geq 0$. The composition $\rho\circ \Psi_n$ is then an irreducible representation of $O_{A_\bn}$ that by construction
sends $s_0$ to a non-zero operator. Therefore, Proposition \ref{int} applies leading to
$\rho\circ \Psi_n=\widetilde{\pi}_{0,z}$ for some $z$ in $\bt$ and we are done.
\end{proof}

The list of all irreducible representations of $\mathcal{W}_\bn$ can be given as well. Indeed, one can reason in much the same way as in the above result noting that the quotient of $\mathcal{W}_\bn$ by the ideal $\widetilde{J}_n$ generated by
$\{P_\Om, A_1, \ldots, A_n\}$ is again (isomorphic with) the whole $O_{A_\bn}$. Therefore, we have the following.

\begin{cor}
If $\pi$ is an irreducible representation of $O_{A_\bn}/ J\cong \mathcal{W}_\bn$, then either
$\pi$ is the Fock representation or there exist $n\geq 1$ and $z$ in $\bt$ such that
$\pi\cong\widetilde{\pi}_{n,z}$.
\end{cor}

\begin{cor}\label{typeI}
$O_{A_\bn}$ is a type-I $C^*$-algebra.
\end{cor}
\begin{proof}
By virtue of a known characterization of type-I $C^*$-algebras given by Sakai in \cite{Sak}, it is enough to make sure that (the image of) any irreducible representation of
$O_{A_\bn}$ contains a non-zero compact operator. This last property is certainly satisfied in the light of
Theorem \ref{catalogue}. Indeed, the projection onto the vacuum belongs to $\widetilde{\pi}_{n, z}(O_{A_\bn})$  for all $z$ in $\bt$ and for all $n\geq 0$.
\end{proof}
\begin{cor}
The weakly monotone $C^*$-algebra $\mathcal{W}_\bn$  is type I.
\end{cor}
\begin{proof}
Being type I  is preserved under taking quotients, and $\mathcal{W}_\bn$ is a quotient of
$O_{A_\bn}$ by Proposition \ref{faithfulrep}.
\end{proof}

In the following, for each $z$ in $\bt$, we denote by $\ch(z)$  the weakly monotone Fock space $\mathcal{F}_{wm}(\ell^2(\bn))$, exactly as in
Proposition \ref{int}.

\begin{lem}\label{multiplicity}
Let $\pi$ be a representation of $O_{A_\bn}$ acting on a separable Hilbert space $\ch_\pi$ with
$\pi(s_0)\neq 0$. There exist a Borel probability
measure $\mu$ on $\bt$, a measurable function $m:\bt\rightarrow \bn\cup\{\infty\}$, and an isometry $W:\int^\oplus_\bt m(z)\ch(z)\, {\rm d}\mu\rightarrow \ch_\pi $ such that:
\begin{itemize}
\item [(i)] $W$ intertwines $\int^\oplus_\bt m(z)\widetilde{\pi}_{0,z} \,{\rm d}\mu$ with $\pi$;
\item [(ii)] the restriction of $\pi(s_0)$ to the orthogonal space of the range of $W$ is zero,
\end{itemize}
where, for every $z\in\bt$,  $m(z)\ch(z)$ is the direct sum of $\ch(z)$ with itself $m(z)$ times and
$m(z)\widetilde{{\pi}}_{0, z}$ is the direct sum of $\widetilde{{\pi}}_{0, z}$ with itself
$m(z)$ times.
\end{lem}

\begin{proof}
It is variation of the proof of Proposition  \ref{int}, of which we will keep the notation.
If $U$ is again the unitary of $\ch_\pi$ obtained as the restriction of
$\pi(s_0)$ to its initial domain (which equals its range) $\mathcal{K}$, we first decompose $\mathcal{K}$ and $U$ as
$\mathcal{K}=\int^\oplus_{\s(U)} \mathcal{K}(z)\,{\rm d}\mu$ and $U=\int^\oplus_{\s(U)} z1_{\ck(z)}\di\mu$, as was done in the above mentioned result. Since $\s(U)$ is contained in $\bt$, we may as well rewrite the decompositions
as $\mathcal{K}=\int^\oplus_{\bt} \mathcal{K}(z)\, {\rm d}\mu$ and $U=\int^\oplus_{\bt} z1_{\ck(z)}\di\mu$, by
setting $ \mathcal{K}(z)=0$ for all $z\in \bt\setminus\s(U)$.
The function $m$ in the statement is then the multiplicity of this decomposition, namely
$m(z):= {\rm dim}\, \mathcal{K}(z)$, for all $z\in\bt$, where $\rm{dim}$ is the dimension of
$\mathcal{K}(z)$ as a Hilbert space, \emph{i.e.} the cardinality of any of its orthonormal bases.\\
It will be convenient to think of each $\mathcal{K}(z)$ as (isometrically isomorphic with)
$\bc^{m(z)}$ whenever $m(z)<\infty$ and as $\ell^2(\bn)$ when $m(z)=\infty$.
For each $z$ in $\bt$, we will also conceive of $\bc^{m(z)}$ as a subspace
of $\ell^2(\bn)$ whose vectors have all
$k$-th components zero if $k>m(z)$.\\
We aim to show that  with these choices of $\mu$ and $m$ the direct integral representation $\int^\oplus_\bt m(z)\widetilde{{\pi}}_{0, z}  \,{\rm d}\mu$  can be realized as a subrepresentation of the given representation $\pi$.\\
In order to define the intertwining
isometry $W$ from $\int^\oplus_\bt m(z)\ch(z)\, {\rm d}\mu$ (recall that $\ch(z)=\mathcal{F}_{wm}(\ell^2(\bn))$, for all
$z$ in $\s(U)$ and $\ch(z)=0$ for all $z$ in $\bt\setminus\s(U)\,$) to $\ch_\pi$, we first need to establish further notation.
Let $\ch^\infty$ denote the countably infinite direct sum of the weakly monotone Fock space
$\mathcal{F}_{wm}(\ell^2(\bn))$ with itself.
For all $z$ in $\bt$ we shall think of the direct sum $m(z)H(z)$ as the subspace of $\ch^\infty$ whose vectors have all
$k$-th components zero if $k>m(z)$.\\
For each $k\geq 1$, we define a measurable function $c_k: \bt\rightarrow\{0,1\}$ as
$c_k(z)=1$ if $k\leq m(z)$ and $c_k(z)=0$ otherwise.
Note that vectors of the type
$$\bt\ni z\mapsto 0\oplus\ldots\oplus \underbrace{c_k(z)f(z)e_\gamma}_{k^{\rm th}}\oplus 0\oplus \ldots\,\, {\rm in}\,\, m(z)\ch(z)$$
span a dense subspace of  $\int^\oplus_\bt m(z)\ch(z)\, {\rm d}\mu$, as $f$ varies in $L^2(\bt, \mu)$.
We are now ready to define $W$ as:
$$ W (0\oplus\ldots\oplus \underbrace{c_k f e_\gamma}_{k^{\rm th}}\oplus 0\oplus \ldots):= \pi(s_\gamma)(0, \ldots, \underbrace{c_k(z)f(z)}_{k^{\rm th}}, 0,  \ldots)\,$$
for all $k\geq 1$, $f$ in $L^2(\bt, \mu)$, $\gamma$ in $\Gamma$ (same notation as in the proof of Proposition \ref{int}).\\
With analogous calculations to those made in the proof of Proposition \ref{int} it is easy to see that $W$ is in fact an isometry and that it intertwines the two representations.\\
As for the second part of the statement, it is enough to note that ${\rm Ran}\, W$ contains $\mathcal{K}$ because
$ W (0\oplus\ldots\oplus \underbrace{c_k f e_\emptyset}_{k^{\rm th}}\oplus 0\oplus \ldots):= (0, \ldots, \underbrace{c_k(z)f(z)}_{k^{\rm th}}, 0,  \ldots)$ for all $f$ in $L^2(\bt, \mu)$.
\end{proof}

\begin{lem}
If $\pi\colon\mathcal{W}_\bn\rightarrow \cb(\ch)$ is a representation such that $\pi(A_0)\neq 0$, then $n\pi_F$ is (unitarily equivalent to) a subrepresentation of $\pi$, where $n:=\dim\rm{Ran}\,\pi(A_0)$.
\end{lem}

\begin{proof}
By hypothesis $E:=\pi(A_0)$ is a nonzero projection. Let $n$ be the dimension of ${\rm Ran}\, E$ as a Hilbert space.
Corresponding to any choice of an orthonormal basis of ${\rm Ran}\, E$, which will have cardinality $n$,  there is a decomposition
of $E$ into a direct sum, which yields the statement exactly as in the proof of Lemma \ref{multiplicity}.
\end{proof}

\begin{thm}\label{dirin}
If $\pi$ is a non-degenerate representation of $O_{A_\bn}$ on a separable Hilbert space, then there exists a sequence
$\{ (\mu_k, m_k ): k\geq 0\}$, with $\mu_k$ being positive (possibly zero) Borel measures on $\bt$ and $m_k: \bt\rightarrow\bn\cup\{\infty\}$ measurable functions for all $k\geq 0$ such that
$$\pi\cong\, \bigoplus_{k\geq 0} \int^\oplus_\bt m_k(z) \widetilde{\pi}_{k, z}\, {\rm d}\mu_k\,,$$
where, for every $z$ in $\bt$, $m_k(z) \widetilde{\pi}_{k, z}$ is
the ampliation of $\widetilde{\pi}_{k, z}$ with multiplicity $m_k(z)$.
\end{thm}

\begin{proof}
Since $\pi$ is non-degenerate, the set $\{i\geq 0: \pi(s_i)\neq 0\}$ is not empty. Let $n$ be its minimum. As
in the proof of Proposition \ref{catalogue}, $\pi$ factors as $\pi=\rho\circ p_n$, where $\rho$ is an irreducible representation of the quotient algebra $O_{A_\bn}/J_n\cong O_{A_\bn}$. Up to this isomorphism,
$\rho$ can be seen as a representation $\pi'$ of $O_{A_\bn}$ such that $\pi'(s_0)\neq 0$.
By applying Lemma \ref{multiplicity}, we then see that $\pi$
must contain a subrepresentation of the type
$\int^\oplus_\bt \widetilde{\pi}_{n, z}\, {\rm d}\mu_n$, for some measure $\mu_n$.
It is now clear how to
go on. Denote by $\ch_n\subset\ch$ the $\pi$-invariant subspace corresponding to the above subrepresentation.
If $\ch_n=\ch_\pi$, there is nothing to do. Otherwise we can consider its orthogonal complement
$\ch_n^\perp$, and repeat the procedure by considering the minimum, say $n'$,  of the set
$\{i\geq 0: \pi(s_i)\upharpoonright_{\ch_n^\perp}\neq 0\}$. Note that $n' > n$.
As in the first part of the proof, there exists
a positive measure $\mu_{n'}$ such that $\int^\oplus_\bt \widetilde{\pi}_{n', z}\, {\rm d}\mu_{n'}$
is a subrepresentation of $\pi_{\ch_n^\perp}$, which acts on $\ch_{n'}\subset \ch_n^\perp$.
If $\ch_n\oplus\ch_{n'}$
exhaust $\ch_\pi$, we are done. If $\ch_n\oplus\ch_{n'}$ is a proper subspace of $\ch_\pi$, then we can reason as above on $(\ch_n\oplus\ch_{n'})^\perp$.
By continuing to reason in this way, we get a (possibly finite) sequence of mutually orthogonal $\pi$-invariant
subspaces, on each of which $\pi$ is (unitarily equivalent to) a direct integral of the type above. By construction, the orthogonal
complement of the direct sum of these subspaces yields a degenerate representation. But then this orthogonal complement
must be zero by hypothesis since $\pi$ is assumed non-degenerate.
\end{proof}

\begin{cor}
If $\pi$ is a non-degenerate representation of $\mathcal{W}_\bn$ on a separable Hilbert space, then there exist $n\in\bn\cup\{\infty\}$ and a sequence $\{ (\mu_k, m_k ): k\geq 1\}$, with $\mu_k$ being positive (possibly zero) Borel measures on $\bt$ and $m_k: \bt\rightarrow\bn\cup\{\infty\}$ measurable functions for all $k\geq 0$ such that
$$\pi\cong\, n\pi_{F}\oplus\bigg(\bigoplus_{k\geq 1} \int^\oplus_\bt m_k(z) \widetilde{\pi}_{k, z}\, {\rm d}\mu_k\bigg)\, .$$
\end{cor}

\subsection{Weakly anti-monotone $C^*$-algebra}\label{anti}
Reversing the order in which the indices appear, one can consider a weakly anti-monotone  Fock space on $\bn$:
$$\mathcal{F}_{wam}(\ell^2(\bn)):=\bigoplus_{k=0}^\infty H_k\,$$
where $H_k$ is the linear span of the set of vectors $e_{i_1}\otimes e_{i_2}\otimes\cdots\otimes e_{i_k}$ with
$k\in\bn$   and $1\leq i_1\leq i_2\leq\ldots\leq i_k$, whereas $H_0=\bc\Om$.\\
On $\mathcal{F}_{wam}(\ell^2(\bn))$ weakly anti-monotone creation (annihilation) operators
$A_i^\dag$ ($A_i$), $i\geq 1$, can be defined  as we did in the previous sections.
Precisely, the action of $A^\dagger_i$ on the orthonormal basis of
$\mathcal{F}_{wam}(\ch)$ is $A^\dagger_i \Om= e_i$ and
\begin{equation*}
A^\dagger_i\,e_{i_1}\otimes e_{i_2}\otimes\cdots\otimes e_{i_k}:=\left\{
\begin{array}{ll}
e_i\otimes e_{i_1}\otimes e_{i_2}\otimes\cdots\otimes e_{i_k}& \text{if}\,\, i\leq i_1 \\
0 & \text{otherwise}, \\
\end{array}
\right.
\end{equation*}
As usual, we denote by $A_i$ the adjoint of $A^\dagger_i$ for all $i\geq 1$.\\
We denote by $\mathcal{W}^a_\bn\subset \cb(\mathcal{F}_{wam}(\ell^2(\bn)))$ the $C^*$-subalgebra generated
by the set $\{A_i, A_i^\dag: i\geq 1\}$ and refer to it as the anti-monotone $C^*$-algebra.\\
Note that $\mathcal{W}^a_\bn$ is a unital $C^*$-algebra, for $A_1A^\dagger_1$ is the identity on
$\mathcal{F}_{wam}(\ell^2(\bn))$.

\begin{prop}
The orthogonal projection onto the vacuum does not belong to $\mathcal{W}^a_\bn$.
\end{prop}

\begin{proof}
It is a straightforward adaptation of the proof of Proposition \ref{vacumnotalg}: it is enough to replace
the $s$ there with the maximum of the set of all indices appearing in the words of $X_n$
(note that now indices appear in increasing order).
\end{proof}
Again, from the fact that the projection onto the vacuum does not sit in $\mathcal{W}^a_\bn$ we see
that  $\mathcal{W}^a_\bn$ cannot be type I.
\begin{cor}\label{nottypeIanti}
$\mathcal{W}^a_\bn$ is not a type-I $C^*$-algebra.
\end{cor}

\section*{Acknowledgments}
\noindent
We would like to thank the anonymous referee for their careful reading and valuable suggestions.\\
The first three authors acknowledge  the support of Italian INDAM-GNAMPA, being
partially supported by Progetto GNAMPA 2023 CUP E53C22001930001 ``Metodi di Algebre di Operatori in Probabilit\`a non Commutativa".    The first three authors are also
supported by Italian PNRR Partenariato Esteso PE4, NQSTI, and Centro Nazionale CN00000013
CUP H93C22000450007.\\
Finally, this research is part of the EU Staff Exchange project 101086394 “Operator
Algebras That One Can See”.

\end{document}